\documentclass[12pt,a4paper]{article}
\usepackage{theorem}
\newtheorem{lemma}{Lemma}

\newtheorem{theorem}[lemma]{Theorem}
\newtheorem{corollary}[lemma]{Corollary}
\newtheorem{exampl}[lemma]{Example}
\newenvironment{example}{\begin{exampl}\upshape}{\hfill\block\end{exampl}}
\newcommand{\C}{{\bf C}}

\newcommand{\E}{{\bf E}}
\newcommand{\F}{{\bf F}}
\newcommand{\G}{{\bf G}}

\newcommand{\Q}{{\bf Q}}
\newcommand{\R}{{\bf R}}

\newcommand{\Z}{{\bf Z}}
\newcommand{\rme}{{\rm e}}

\newcommand{\cA}{{\cal A}}
\newcommand{\cB}{{\cal B}}
\newcommand{\cC}{{\cal C}}

\newcommand{\cF}{{\cal F}}
\newcommand{\cG}{{\cal G}}
\newcommand{\cH}{{\cal H}}
\newcommand{\cJ}{{\cal J}}

\newcommand{\cL}{{\cal L}}

\newcommand{\cP}{{\cal P}}
\newcommand{\cQ}{{\cal Q}}
\newcommand{\cR}{{\cal R}}

\newcommand{\alp}{\alpha}

\newcommand{\lam}{\lambda}
\newcommand{\del}{\delta}

\newcommand{\tr}{{\rm tr}}

\newcommand{\spec}{{\rm spec}}

\newcommand{\norm}{\Vert}

\newcommand{\supp}{{\rm supp}}

\newcommand{\hash}{\#}

\newcommand{\Schrodinger}{Schr\"odinger }
\newcommand{\implies}{\Rightarrow}
\newcommand{\pr}{\prime}
\newcommand{\til}{\widetilde}
\newcommand{\wdh}{\widehat}
\newcommand{\eqref}[1]{(\ref{#1})}
\newcommand{\block}{\hfill\rule{2.5mm}{2.5mm}}
\newenvironment{proof}{\textbf{Proof}}{\block}
\newenvironment{choices}{ \left\{ \begin{array}{ll} }{\end{array}\right.}
\newcommand{\mtrx}[1]{\left(\begin{array}{cc}#1 \end{array}\right)}
\newcommand{\txtmtrx}[1]{\raisebox{0.25ex}{\scalebox{0.6}{$\mtrx{#1}$}}}
\newcommand{\vctr}[1]{\left(\begin{array}{c}#1 \end{array}\right)}
\newcommand{\txtif}{\mbox{ if }}
\newcommand{\txtotherwise}{\mbox{ otherwise}}

\usepackage{hyperref}
\usepackage{graphicx}
\title{Algebraic Aspects of Spectral Theory}
\author{E B Davies}
\date{8 June 2010}
\begin{document}
\maketitle
%%%%%%%%%%%%%%%%%%%%%
\begin{abstract}
We describe some aspects of spectral theory that involve algebraic
considerations but need no analysis. Some of the important
applications of the results are to the algebra of $n\times n$
matrices with entries that are polynomials or more general analytic
functions.
\end{abstract}

Short title: Algebraic Spectral Theory

\section{Introduction}\label{intro}

This paper describes a range of results in spectral theory that can
be formulated and proved in an algebraic context without assuming
that the algebra has a norm or even that the underlying scalar field
is the field of complex numbers. In some of the applications the
`scalars' involved are polynomials or rational functions. The author
undertook this research in order to obtain some insights into a
range of contexts in which such ideas are used, including some
related to polynomial pencils, so-called non-linear spectral theory
and control theory. Such problems are often not easily formulated in
traditional Banach algebra terms. As well as its potential applications, indicated below, this study helps to delineate those aspects of spectral theory that do not depend on analytic tools -- norms, continuity, complex analysis, etc.

A very limited account of spectral theory at our degree of
generality may be found in \cite[pp. 1-12, 87]{bourbaki}. A fuller
treatment is given in \cite[Chap.~1]{dales}, but it still has little
overlap with the contents of this paper. More relevant is
\cite[Appendix~4]{whitney}, which gives an extensive account of
various rings and fields of polynomials, analytic and meromorphic
functions from the algebraic perspective. Some elements of this
theory are described in the final section of the present paper.
Simple examples of algebras to which the theorems in this paper
might be applied are

\begin{enumerate}

\item
Algebras of bounded or unbounded operators on a Hilbert space and
infinite tensor product algebras.

\item
An important application is to the algebra $M(n,\cB)$ of all
$n\times n$ matrices with entries in a commutative algebra $\cB$,
particularly the case in which $\cB$ is the ring of polynomials (or
rational functions) in a single complex variable. This example
features in the theory of loop groups, \cite[Sect.~3.5]{PS} and in
the study of evolution equations that involve matrices with
coefficients that depend polynomially or rationally on time.

\item
The space $\cH_{\infty}^n$ of all bounded, analytic, $n\times n$ matrix-valued functions on the upper half plane or the unit disc plays a key role in optimal control theory; see \cite{LR,ZDG}. References to the early literature on bounded operator-valued functions on the unit disc may be found in \cite[Chap.~5]{S-NF}. In both cases one may regard the coefficient ring as the integral domain consisting of all bounded, complex-valued analytic functions on a domain in $\C$; the quotient field is a certain class of meromorphic functions on the same domain.

\item
Voiculescu has developed an algebraic approach to non-commutative
probability theory that has substantial applications to random
matrix theory; see \cite{voic} and \cite{tao}, where the term `free
probability theory' is used. In this application one considers
$n\times n$ matrices whose entries lie in a certain algebra of
unbounded random variables by using highly algebraic methods.
\end{enumerate}

We emphasize that much of the material in this paper is classical if
$\cA$ is the algebra of all $n\times n$ complex matrices; see
\cite{GLR}. We are mainly concerned with the extent to which the
results hold for more general choices of the algebra and base field.
The final section describes some examples of relevant fields. We
give a number of basic algebraic definitions in the paper for the
benefit of those who are not familiar with the subject.

\section{Some elementary results}

Let $\cA$ be an algebra over a field $\F$; we always assume that
$\cA$ contains an identity element, which we denote by $e$. We say
that $\lam\in\F$ does not lie in the spectrum of $a\in\cA$ if there
exists $b\in\cA$ such that $(a-\lam e)b=b(a-\lam e)=e$. The spectrum
of $a$ might be empty or equal to $\F$.

The formula $L_a(x)=ax$ defines a one-one algebra homomorphism from
$\cA$ to the algebra $\cL(\cA)$ of linear operators on $\cA$. The
similar formula $R_a(x)=xa$ defines an anti-homomorphism, in the
sense that $R_aR_b=R_{ba}$ for all $a,\, b\in\cA$. The following
lemma has some similarity to a corresponding result for C*-algebras,
but it is applicable much more generally, for example to the algebra
of $n\times n$ upper triangular matrices with entries in any field.

\begin{lemma}

We have $L_aR_b=R_bL_a$ for every $a,\, b\in \cA$. If $M\in
\cL(\cA)$ and $MR_b=R_bM$ for all $b\in\cA$ then there exists
$a\in\cA$ such that $M=L_a$.
\end{lemma}

\begin{proof}
The first statement is elementary. To prove the second we put
$a=M(e)$ and observe that
\[
M(b)=M(R_be)=R_bM(e)=M(e)b=ab
\]
for all $b\in\cA$, so $M=L_a$.
\end{proof}

\begin{lemma}\label{lemmainv}
If $L_a$ is invertible as an element of $\cL(\cA)$ then $a$ is an
invertible element of $\cA$ and $(L_a)^{-1}=L_{a^{-1}}$. Hence
\[
\spec(L_a)=\spec(a)
\]
for all $a\in\cA$.
\end{lemma}

\begin{proof}
Let $M\in\cL(\cA)$ and
\begin{equation}
ML_a(x)=x=L_aM(x)\label{Lainv}
\end{equation}
for all $x\in \cA$. Putting $M(e)=b$ and $x=e$ in the second
equality of (\ref{Lainv}) yields $e=ab$. Putting $x=bs$ in the first
equality of (\ref{Lainv}) yields $M(s)=M(abs)=ML_a(bs)=bs$ for all
$s\in\cA$. Therefore $M=L_b$. Finally putting $x=e$ in the first
equality of (\ref{Lainv}) yields $ba=L_bL_ae=e$ so $b=a^{-1}$.
\end{proof}

We say that a non-zero element $a$ in an algebra $\cA$ over $\F$ is
\emph{algebraic} if there exists a non-zero polynomial $p$ with
coefficients in $\F$ such that $p(a)=0$. The set of all such
polynomials is a non-zero, proper ideal $\cJ$ in the algebra $\cP$
of all polynomials, and contains a unique monic polynomial $m$ of
lowest degree, called the minimum polynomial of $a$. Moreover
$\cJ=\{ mq: q\in\cP\}$. If $\cA$ is finite-dimensional then all
$a\in\cA$ are algebraic, but there are important
infinite-dimensional algebras in which the same holds.

\begin{example}\label{tensor}
Let $\cA_r=M(n_r,\F)$ for all positive integers $r$ and put
$\cA=\bigotimes_{r=1}^\infty \cA_r$. The algebra $\cA$ is generated
by expressions $x=\otimes_{r=1}^\infty x_r$ where all but a finite
number of the $x_r$ equal the identity element of $\cA_r$.
Equivalently $\cA=\bigcup_{m=1}^\infty\cB_m$ where
\[
\cB_m=\bigotimes_{r=1}^m \cA_r\simeq M(n_1n_2\ldots n_m,\F).
\]
Clearly every element of $\cA$ is algebraic, but there is no upper
bound on the degrees of the minimum polynomials.
\end{example}

\begin{theorem}\label{alginvert}
Let $a$ be an algebraic element of the subalgebra $\cA$ of the
algebra $\cB$, where $\cA$ and $\cB$ have the same identity element
$e$. If $a$ has a left or right inverse $b$ in $\cB$ then it has a
two-sided inverse in $\cA$. The spectrum of $a$ in either $\cA$ or
to $\cB$ equals the set of zeros of its minimum polynomial, and must
be finite.
\end{theorem}

\begin{proof} Let $m$ be the minimum polynomial of $a$ and assume that $a$ has a right inverse $b\in\cB$. If $m(\lam)=\lam^n+\sum_{r=0}^{n-1} \alp_r \lam^r$ and $\alp_0=0$ then one would have
\begin{eqnarray*}
\lefteqn{a^{n-1}+\alp_{n-1}a^{n-2}+\ldots + \alp_2a+\alp_1e}\\
&=&\left( a^{n-1}+\alp_{n-1}a^{n-2}+\ldots +\alp_2a+ \alp_1e\right)ab\\
&=& m(a)b\\
&=&0.
\end{eqnarray*}
This contradicts the assumption that $m$ is the minimum polynomial,
so we deduce that $\alp_0\not= 0$. The same conclusion holds if we
assume that $a$ has a left inverse in $\cB$. It now follows
immediately that
\[
a^{-1}=-\alp_0^{-1}\left( a^{n-1}+\alp_{n-1}a^{n-2}+\ldots +
\alp_1e\right)\in \cA.
\]
The above calculations show that $a$ is invertible (in either $\cA$
or $\cB$) if and only if $m(0)\not=0$. Since the minimum polynomial
of $a-\lam e$ is $m_\lam(s)=m(s+\lam)$, we conclude that
$\lam\in\spec(a)$ if and only if $m(\lam)=0$.
\end{proof}

It is sometimes useful to distinguish between the left and right
spectra of an element $a$ of the algebra $\cA$. One says that
$\lam\in\spec_L(a)$ if there does not exist $b\in\cA$ such that
$b(\lam e-a)=e$, and that $\lam\in\spec_R(a)$ if there does not
exist $c\in\cA$ such that $(\lam e-a)c=e$. One sees immediately that
\[
\spec(a)=\spec_L(a)\cup\spec_R(a).
\]

The elements $b,\, c$ above need not be unique. Indeed, if there is
a unique solution $b\in\cA$ of $ab=e$ then the identity
\[
a(b+ba-e)=ab+(ab)a-a=e
\]
implies that $b$ is the two-sided inverse of $a$.

\begin{lemma}\label{LRspeclemma}
If $a$ is an algebraic element of $\cA$ then
\[
\spec_L(a)=\spec_R(a)=\spec(a)
\]
for all $a\in\cA$.
\end{lemma}

\begin{proof}
Let $m$ be the minimum polynomial of $\lam e -a$. If $\lam\notin
\spec_L(a)$ then there exists $b\in\cA$ such that $b(\lam e-a)=e$.
Since $\lam e-a$ is algebraic, Theorem~\ref{alginvert} implies that
$\lam\notin \spec(a)$. Hence $\spec(a)\subseteq \spec_L(a)$. The
inclusion $\spec(a)\subseteq \spec_R(a)$ has a similar proof and the
reverse inclusions are elementary.
\end{proof}

\begin{example}
The condition in Lemma~\ref{LRspeclemma} that the element $a$ is
algebraic cannot be omitted. Let $\cA$ be the algebra of all bounded
linear operators on the Hilbert space $\ell^2(\Z^+)$ and let
$a\in\cA$ be defined by $(a\phi)(n)=\phi(n+1)$ for all $\phi \in
\ell^2(\Z^+)$ and all $n\geq 1$. Then
\begin{eqnarray*}
\spec_L(a)&=&\{ z\in\C: |z|\leq 1\},\\
\spec_R(a)&=&\{ z\in\C: |z|= 1\}.
\end{eqnarray*}
The proof of the first identity uses $(\lam e-a)\phi_\lam=0$
whenever $|\lam |<1$, where $\phi_\lam\in \ell^2(\Z^+)$ is defined
by $\phi_\lam(n)=\lam ^n$ for all $n\geq 1$. The proof of the second
identity uses $aa^\ast=e$ and $ (\lam e-a)c_\lam=e$ whenever $|\lam
|<1$, where $c_\lam =-a^\ast(e-\lam a^\ast)^{-1}$.
\end{example}

We next prove a spectral mapping theorem for polynomials.

\begin{theorem}
If the field $\F$ is algebraically closed, $a\in \cA$ and $p$ is a
polynomial then
\[
\spec(p(a))=\{p(z):z\in \spec(a)\}.
\]
\end{theorem}

\begin{proof}
Suppose that $p$ is monic and of degree $n$. The statement of the
theorem can be rewritten in several ways each of which is equivalent
to the next one.
\begin{itemize}
\item
Given $y\in\F$ then $y\in\spec(p(a))$ if and only if there exists
$x\in\spec(a)$ such that $y=p(x)$.

\item Let
\[
q(s)=p(s)-y=\prod_{r=1}^n (s-s_r),
\]
where $s_r$ depend on $y$. Then $0\in \spec(q(a))$ if and only if
there exists $x\in\spec(a)$ such that $q(x)=0$.

\item
$0\in \spec(q(a))$ if and only if there exists $x\in\spec(a)$ and
$r\in\{1,\ldots,n\}$ such that $x-s_r=0$.

\item
$0\in \spec(q(a))$ if and only if there exists $r\in\{1,\ldots,n\}$
such that $s_r\in\spec(a)$.

\item
$q(a)$ is not invertible if and only if there exists
$r\in\{1,\ldots,n\}$ such that $a-s_re$ is not invertible.

\item
$q(a)$ is invertible if and only if  $a-s_re$ is  invertible for all
$r\in\{1,\ldots,n\}$.
\end{itemize}
The truth of the final statement is an elementary consequence of the
fact that $q(a)$ is the product of the $a-s_re$, all of which
commute in $\cA$.
\end{proof}

Without the hypothesis that $\F$ is algebraically closed, the
theorem is false. For example if $\F=\R$ and $a=\txtmtrx{1&1\\-1&1}$
then $\spec(a)=\emptyset$ but $\spec(a^4)=\{-4\}$.

\section{Commutative Algebras}

In this section we present an algebraic version of Gelfand's
representation of commutative Banach algebras as function algebras.

Let $\cA$ be a commutative algebra with identity $e$ over a field
$\F$ and let $J_m$ denote its maximal proper ideals, which
parametrized by $m\in M$. For each $m\in M$, $\F_m=\cA/J_m$ is a
commutative algebra with no proper ideals and hence is a field
containing $\F$. The natural homomorphism $\phi_m:\cA\to \F_m$ is
called a character of $\cA$.

The set $M$ of all maximal ideals in $\cA$ is called the
max-spectrum of $\cA$ in algebraic geometry to distinguish it from
the spectrum, namely the set $N$ of all prime ideals. Elements of
$M$ correspond to algebra homomorphisms mapping $\cA$ onto a field,
which we call characters, while elements of $N$ correspond to
algebra homomorphisms mapping $\cA$ into a field, or equivalently
onto an integral domain. Every maximal ideal is obviously prime.

Determining the maximal ideals, prime ideals or characters in an
algebra is a highly non-trivial matter. Hilbert's Nullstellensatz
implies that if $\cA$ is the algebra of polynomials on $\C^n$ then
every character is obtained by evaluation at some point of $\C^n$;
see \cite[p.~85]{AM}. For some other algebras of unbounded
continuous functions this is not true and the question is much
harder.

The Gelfand formula
\[
\wdh{a}(m)=\phi_m (a)
\]
takes an element $a\in\cA$ and maps it by an algebra homomorphism to
$\wdh{a}\in \cF(M)$, where $\cF(M)$ is the algebra of all functions
$f$ on $M$ such that $f(m)\in\F_m$. (More correctly $\cF(M)$ is the
algebra of sections of a certain bundle over $M$.)

\begin{theorem}\label{comminv}
The element $a\in\cA$ is invertible if and only if $\wdh{a}(m)\not=
0$ for all $m\in M$. If $a\in\cA$ then
\[
\spec(a)=\F\cap\{ \wdh{a}(m):m\in M\}.
\]
\end{theorem}

\begin{proof}
If $a$ is invertible and $ab=e$ then
\[
\wdh{a}(m)\wdh{b}(m)=\phi_m(a)\phi_m(b)=\phi_m(ab)=e,
\]
so $\wdh{a}(m)\not= 0$ for every $m\in M$. Conversely if $a$ is not
invertible then $a\cA$ is a proper ideal in $\cA$ and is contained
in some maximal proper ideal $J_m$. Since $a\in J_m$ one has
$\wdh{a}(m)=0$.

The final statement of the theorem is proved by applying the above
to $a-\lam e$ for every choice of $\lam\in \F$.
\end{proof}

\begin{example}
Let $I=[0,1]$ and let $\cA_1$ be the algebra over $\Q$ consisting of
all continuous functions $f:I\to \R$ such that $f(\Q\cap I)\subseteq
\Q$, where $\Q$ is the field of rational numbers. Given a positive
integer $n$ and any sequence $a_0,a_1,\ldots , a_n\in\Q$ one may
define a function $f\in \cA_1$ by linear interpolation starting from
the values $f(r/n)=a_r$ where $0\leq r\leq n$. This construction may
be used to prove that $\cA_1$ is uniformly dense in the algebra of
all continuous real-valued functions on $I$.

For each $a\in I$ one may define the character $\phi_a$ on $\cA_1$
by
\begin{equation}
\phi_a(f)=f(a).\label{commchar}
\end{equation}
One notes that $\F_a=\Q$ if $a\in \Q\cap I$ but $\F_a=\R$ otherwise.
\end{example}

\begin{theorem}
Every character $\phi$ of $\cA_1$ is of the form $\phi_a$ for some
$a\in I$.
\end{theorem}

\begin{proof}
Suppose that $\phi$ is not of the stated form. For every $a\in I$
the kernels $J$ and $J_a$ of $\phi$ and $\phi_a$ respectively are
maximal proper ideals, so there  exists $f_a\in J$ such that
$f_a(a)\not= 0$. By a compactness argument there exist $a_1,\ldots,
a_m\in I$ such that $h:=\sum_{r=1}^m |f_{a_r}|^2>0$ on $I$. This
yields a contradiction, because $h\in J$ and a proper ideal cannot
contain an invertible element.
\end{proof}

\begin{example}
The algebra $\cA_2=C_b(\R)$ of all bounded, continuous,
complex-valued functions $f$ on $\R$ is a C*-algebra for the norm
\[
\norm f\norm =\sup\{ |f(x)|: x\in \R\},
\]
and for no other norm; see \cite[Theorem~3.2.6]{dales}. The
characters of $\cA_2$ are all norm continuous and are in one-one
correspondence with the points in the Stone-\v{C}ech
compactification of $\R$ and all characters are complex-valued; see
\cite[Def.~ 4.2.6]{dales}. The only characters that can be
identified constructively are those given by the formula
(\ref{commchar}) where $a\in \R$.
\end{example}

\begin{example}
The algebra $\cA_3= C(\R)$ of all continuous, complex-valued
functions on $\R$ has `finite' characters given by (\ref{commchar}).
We say that a character $\phi$ on $\cA_3$ is `at infinity' if it is
not of this form. Lemma~\ref{inftychar1} and
Theorem~\ref{inftychar2} give some insight into the form of the
characters at infinity.
\end{example}

\begin{lemma}\label{inftychar1}
Every character $\phi$ at infinity on $\cA_3$ has a kernel $J_\phi$
that contains the ideal
\begin{equation}
J=\{ f\in \cA_3: \supp(f) \mbox{ is compact }\}.\label{Jideal}
\end{equation}
\end{lemma}

\begin{proof}
Suppose that $f\in\cA_3$ has support contained in $(-n,n)$ for some
$n$. Let $a\in [-n,n]$. Since $J_\phi$ and $J_a=\{
f\in\cA_3:f(a)=0\}$ are both maximal proper ideals, there exists
$f_a\in J_\phi$ such that $f_a(a)\not= 0$. By a compactness argument
there exist $a_1,\ldots, a_m\in [-n,n]$ such that $h:=\sum_{r=1}^m
|f_{a_r}|^2>0$ on $[-n,n]$. Since $f=hk$ for some $k\in \cA_3$, it
follows that $f\in J_\phi$. This proves the lemma.
\end{proof}

\begin{theorem}\label{inftychar2}
If $\phi$ is a character of $\cA_3$ at infinity then the field
$\phi(\cA)$ contains the field $G$ of all complex-valued rational
functions.
\end{theorem}

\begin{proof}
Every $g\in G$ has a finite number of poles, so there exists a
smallest integer $n\geq 0$ such that $g(x)$ is continuous on
$\{x\in\R:|x|\geq n\}$. We define $\rho(g)\in\cA_3$ to be equal to
$g$ on $\{x\in\R:|x|\geq n\}$ and by linear interpolation in
$[-n,n]$.

If $J$ is defined by (\ref{Jideal}), then
\[
\rho(g_1)+\rho(g_2)-\rho(g_1+g_2)\in J ,\hspace{2em}
\rho(g_1)\rho(g_2)-\rho(g_1g_2)\in J
\]
for all $g_1,\, g_2\in G$. Therefore $g\to \rho(g)+J$ defines a
one-one structure preserving map from $G$ into $\cA/J$. By an abuse
of notation one may say that $G\subseteq \cA/J$.

There is a one-one correspondence between ideals in $\cA/J$ and
ideals in $\cA$ that contain $J$. Similarly, the characters on $\cA$
whose kernels contain $J$ may be identified with the characters on
the algebra $\cA/J$. The restriction of such characters to $G$,
regarded as a subfield of $\cA/J$, must be one-one, because a field
contains no proper ideals, so the field $\phi(\cA/J)$ must contain a
copy of $G$.
\end{proof}

\section{Resolvent families}

We now turn to the study of (pseudo) resolvent families; the ideas
here are taken from \cite{EBD2}. From this point onwards $\cA$ will
denote a (typically non-commutative) algebra over a field $\F$. A
resolvent family $(r,S)$ in $\cA$ consists of a set $S\subseteq \F$
together with a map $r:S\to \cA$ such that
\begin{equation}
r_\lam-r_\mu=(\mu-\lam)r_\lam r_\mu\label{reseq}
\end{equation}
for all $\lam,\, \mu\in S$. By interchanging the role of $\lam,\,
\mu$ one readily sees that this implies that $r_\lam$ and $r_\mu$
commute for all $\lam,\, \mu \in S$. We say that $(\til{r},\til{S})$
extends $(r,S)$ if $S\subseteq \til{S}$ and $\til{r}(\lam)=r(\lam)$
for all $\lam\in S$. We say that $(r,S)$ is maximal if it has no
proper extension.

If $a\in \cA$ and $S=\F\backslash \spec(a)$ then one obtains a
(true) resolvent family by putting $r_\lam=(\lam e-a)^{-1}$ for all
$\lam\in S$. (Pseudo) resolvent families are not all of this type;
if $a$ is an unbounded closed linear operator on a Banach space
$\cB$ and $\cA$ is the algebra of all bounded operators on $\cB$
then one may define the resolvent family associated with $a$ just as
in the case in which $a$ is bounded. The theory of such resolvent
families is very highly developed because of the role it plays in
the theory of differential operators and one-parameter semigroups,
\cite{LOTS}.

\begin{theorem}
A resolvent family $(r,S)$ is the family associated with an element
$a\in\cA$ if and only if it is maximal and $r_\lam$ is invertible
for some (equivalently all) $\lam\in S$. If this holds then
$\spec(a)=\F\backslash S$.
\end{theorem}

\begin{proof}
Suppose that $br_\mu=r_\mu b =e$ for some $b\in\cA$ and $\mu\in S$.
Multiplying \eqref{reseq} on the left (right) by $b$ yields $e=
(b-\mu e+\lam e) r_\lam =r_\lam (b-\mu e+\lam e)$, so $r_\lam$ is
invertible for all $\lam\in S$. Putting $a=\mu e-b$ then yields
$r_\lam=(\lam e-a)^{-1}$ for all $\lam\in S$.

Given $a\in\cA$ let $(r_0,S_0)$ be the resolvent family associated
with $a$, where $S_0=\F\backslash \spec(a)$. Suppose that it is not
maximal, that $(r,S)$ is a proper extension and that $\lam\in S\cap
\spec(a)$. Then $r_\lam$ is invertible with inverse which we denote
$b$ and $r_\mu=(\mu e-a)^{-1}$ for some chosen $\mu\in S$. If we
multiply \eqref{reseq} on the left by $b$ and on the right by $\mu
e-a$ we obtain
\[
\mu e-a -b=(\mu-\lam)e.
\]
Therefore $b=\lam e-a$; since $b$ is invertible we conclude that
$\lam\notin \spec(a)$. The contradiction implies that $(r_0,S_0)$ is
maximal.
\end{proof}

\begin{theorem}\label{unique}
Every resolvent family $(r,S)$ is uniquely determined on $S$ by
$r_\mu$ for a single chosen $\mu\in S$. Moreover $(r,S)$ has a
unique maximal extension $(\til{r},\til{S})$.
\end{theorem}

\begin{proof}
Given $\alp,\, \lam\in  S$ the resolvent equation implies that
\[
(e+(\lam-\alp)r_\alp)(e+(\alp-\lam)r_\lam)=e.
\]
Therefore $e+(\lam-\alp)r_\alp$ is invertible and
\[
r_\lam=(\alp-\lam)^{-1}\left\{ (e+(\lam-\alp)r_\alp)^{-1} -e\right\}=%
(e+(\lam-\alp)r_\alp)^{-1}r_\alp.
\]
If we define
\[
\til{r}_\lam= (e+(\lam-\alp)r_\alp)^{-1}r_\alp
\]
on the set $\til{S}$ of all $\lam\in\F$ for which
$e+(\lam-\alp)r_\alp$ is invertible, then $\til{r}$ must be the
maximal extension of $r$ provided $\til{r}$ satisfies the resolvent
equation. If $\lam,\, \mu\in \til{S}$ then
\begin{eqnarray*}
(\mu-\lam)\til{r}_\lam
\til{r}_\mu&=&(\mu-\lam)r_\alp(e+(\lam-\alp)r_\alp)^{-1}%
(e+(\lam-\alp)r_\alp)^{-1}r_\alp\\
&=&\left\{ (e+(\mu-\alp)r_\alp)-(e+(\lam-\alp)r_\alp)\right\}\\
&&.(e+(\lam-\alp)r_\alp)^{-1} (e+(\mu-\alp)r_\alp)^{-1}r_\alp\\
&=&\left\{(e+(\lam-\alp)r_\alp)^{-1}-
(e+(\mu-\alp)r_\alp)^{-1}\right\}r_\alp\\
&=& \til{r}_\lam-\til{r}_\mu.
\end{eqnarray*}
\end{proof}

Motivated by the last two theorems, we define the spectrum
$\spec(r,S)$ of a resolvent family $(r,S)$ to be $\F\backslash
\til{S}$ where $(\til{r},\til{S})$ is its unique maximal extension.
Our next task is to explain how one can use algebraic ideas to
classifying the spectrum into parts.

If $\cJ$ is a two-sided ideal in $\cA$, there is a natural algebra
homomorphism $\pi_\cJ$ from $\cA$ onto the quotient algebra
$\cA/\cJ$. Given $a\in\cA$ we put
\begin{eqnarray*}
\spec_\cJ(a)&=&\spec(\pi_\cJ(a)),\\
\spec_{L,\cJ}(a)&=&\spec_L(\pi_\cJ(a)),\\
\spec_{R,\cJ}(a)&=&\spec_R(\pi_\cJ(a)).
\end{eqnarray*}
See \cite[Chap.~1.3]{EE} for a full discussion of the various
definitions of the essential spectrum of an operator on a Hilbert
space $\cH$; one widely used definition arises by taking $\cA$ to be
the algebra of all bounded operators on $\cH$ and $\cJ$ to be the
ideal of all compact operators. The quotient algebra $\cA/\cJ$ is
called the Calkin algebra and arises in the study of Fredholm
operators and index theory.

\begin{lemma}
Let $(r,S)$ be a maximal resolvent family in $\cA$. If one defines
\[
\spec_\cJ(r,S)=\spec(\pi_\cJ r,S)
\]
then
\[
\spec_\cJ(r,S)\subseteq \spec(r,S).
\]
\end{lemma}

\begin{proof}
One sees immediately that $(\pi_\cJ r,S)$ is a resolvent family in
$\cA/\cJ$. If $\til{S}$ is the domain of its maximal extension then
the lemma is simply a restatement of the fact that $S\subseteq
\til{S}$.
\end{proof}

Note that if $(r,S)$ is associated with a closed unbounded operator
$a$ on $\cH$, then the resolvent family $(\pi_\cJ r,S)$ has a much
less direct relationship with $a$ because $\pi_\cJ a$ is not easy to
define.

We turn to some applications to perturbation theory. The algebra
$\cL(\cH)$ of all bounded operators on a Hilbert space $\cH$
contains many two-sided ideals, but the set $\cC_\infty$ of all
compact operators and the set $\cC_1$ of all trace class operators
are particularly important. The following results are abstractions
of theorems known in the above context; they have proved important
in quantum theory, because of the existence of physically relevant
$C^\ast$-algebras with many two-sided ideals. See
\cite{EBD2,GI1,GI2} for many references to these recent
developments. These ideals allow the division of the classical
essential spectrum into parts with different natures.

Let $\cJ$ be a two-sided ideal in an algebra $\cA$ over the field
$\F$. If $a,\, b\in\cA$ and $\lam\notin \spec(a)\cup\spec(b)$ then
the identity
\[
(\lam e-a)^{-1}-(\lam e-b)^{-1}=(\lam e-a)^{-1}(a-b)(\lam e-b)^{-1}
\]
implies that $a-b\in\cJ$ if and only if the difference of the two
resolvents lies in $\cJ$. Moreover the definition of the
$\cJ$-spectrum implies immediately that $\spec_\cJ(a)=\spec_\cJ(b)$
provided $a-b\in\cJ$.

If $A,\, B$ are two unbounded closed operators on a Hilbert space
$\cH$  the analogous result is proved by treating the resolvent
operators directly as follows; see \cite[Theorem~IX.2.4]{EE}, where
$\cJ$ is the ideal of compact operators in $\cL(\cH)$.

\begin{theorem}
Let $(r_1,S_1)$ and $(r_2,S_2)$ be two resolvent families in $\cA$
and let $S=S_1\cap S_2$. If $r_{1,\lam}- r_{2,\lam}$ lies in the
two-sided ideal $\cJ$ for some $\lam\in S$ then this holds for all
$\lam\in S$ and $\spec_\cJ(r_1)=\spec_\cJ(r_2)$.
\end{theorem}

\begin{proof}
If we put $\til{r}_i=\pi_\cJ (r_i)$ then the hypothesis implies that
$\til{r}_{1,\lam}=\til{r}_{2,\lam}$ for some $\lam\in S$.
Theorem~\ref{unique} now implies that $\til{r}_{1}=\til{r}_{2}$
throughout $S$. Therefore they have the same maximal extension and
the original resolvents have the same $\cJ$-spectrum.
\end{proof}

The following theorem is adapted from \cite[Lemma~10.1.8]{peller},
where applications to the essential spectra of Hankel operators are
given.

\begin{theorem}\label{specunion}
If $a_1,\ldots, a_n\in \cA$ and $a_ra_s\in\cJ$ for all $r\not= s$
then
\[
\spec_\cJ(a)\cup\{0\}=\bigcup_{r=1}^n\spec_\cJ(a_r)
\]
where $a=\sum_{r=1}^n a_r$.
\end{theorem}

\begin{proof}
We put  $\til{x}=\pi_\cJ x\in\cA/\cJ$ for every $x\in\cA$. The
hypotheses imply that $\til{a}_r\til{a}_s=0$ for all $r\not= s$ and
hence that
\[
(\lam \til{e}-\til{a}_1)\ldots(\lam
\til{e}-\til{a}_n)=\lam^{n-1}(\lam\til{e}-\til{a})
\]
for all $\lam\in\F$. Since all the terms in this equation commute we
conclude that if $\lam\not= 0$ then $\lam\til{e}-\til{a}$ is
invertible if and only if $\lam \til{e}-\til{a}_r$ are invertible
for all $r$. This completes the proof as far as non-zero $\lam$ are
concerned.

If $0\notin \spec(\til{a}_1)$ then $\til{a}_1$ is invertible and the
identity $\til{a}_1\til{a}_r=0$ implies that $\til{a}_r=0$ for all
$r\not= 1$. Therefore $0\in\bigcup_{r=1}^n\spec(\til{a}_r)$ in all
cases.
\end{proof}

\begin{example}
The following application of Theorem~\ref{specunion} is well-known
in the spectral analysis of \Schrodinger operators. Let
$p_1,\ldots,p_n\in\cA$ satisfy $p_rp_s=\del_{r,s}p_r$ for all $r,\,
s$ and $\sum_{r=1}^np_r=e$. Let $a\in\cA$ and assume that
$a_{r,s}:=p_rap_s\in\cJ$ for all $r\not= s$. Then by applying
Theorem~\ref{specunion} to the double sum $a=\sum_{r,s}a_{r,s}$ one
obtains
\[
\spec_\cJ(a)\cup\{ 0\}= \bigcup_{r=1}^n \spec_\cJ(a_{r,r}).
\]
\end{example}

\begin{example}
Let $\cA$ be the algebra $M(n,\cP)$ of all $n\times n$ matrices
whose entries are polynomials in a variable $x\in \F$. The spectrum
of an element of $\cA$ is normally equal to $\F$ because an element
of $\cA$ is invertible if and only if its determinant is a non-zero
constant. However, if $n=2$ and
\[
a=\mtrx{1&1\\ x&1+x},\hspace{2em} b=\mtrx{0&1+x^3\\ 0&0}
\]
$x\in\F$ being the polynomial variable, then $\spec(a)=\F\backslash
\{ 0\}$ while $\spec(b)=\{ 0\}$.

By passing to suitable quotient algebras one obtains more
interesting spectra. Let $\cA=M(n,\cP)$ and let $S$ be a finite
subset of $\F$. We define $\cB_S$ to be the algebra of all functions
$f:S\to M(n,\F)$ with the obvious pointwise operations and define
$\pi_S:\cA\to \cB_S$ by $\pi_S(a)= f$ where $f(s)=a(s)$ for all
$s\in S$. The kernel of $\pi_S$ is a two-sided ideal $\cJ_S$ in
$\cA$ and $\pi_S$ is an algebra homomorphism from $\cA$ onto
$\cB_S$. Moreover
\[
\spec_{\cJ_S}(a)=\spec(\pi_Sa)=\bigcup_{s\in S}\spec(a(s)).
\]
for all $a\in\cA$. Hence $\hash(\spec_{\cJ_S}(a))\leq n(\hash S)$.
As $S$ increases, the kernels of the ideals $\cJ_S$ decrease and
$\spec_{\cJ_S}(a)$ increases to $\bigcup_{x\in \F}\spec(a(x))$.

If one puts $\F=\R$, $n=2$ and defines $a\in\cA$ by
\[
a=\mtrx{ x^2 &1\\ 0&x^2},
\]
then $\spec(a)=\R$ but $\bigcup_{x\in \R}\spec(a(x))=[0,\infty)$.
\end{example}

\section{Connections between $ab$ and $ba$}

If $a,\, b$ lie in an algebra $\cA$ over $\F$, an elementary
calculation shows that
\[
bp(ab)=p(ba)b
\]
for every polynomial $p$ with coefficients in $\F$. The calculations
on this section are based on the idea that this relationship might
be valid for other classes of function. The following theorem may be
found in \cite[Prop.~1.5.29]{dales}.

\begin{lemma}\label{alginv}
If $a,b\in\cA$ then
\[
\spec(ab)\cup\{0\}=\spec(ba)\cup\{0\}.
\]
If $\dim(\cA)<\infty$ then
\[
\spec(ab)=\spec(ba).
\]
Moreover $ab$ is invertible if and only if both $a$ and $b$ are
invertible.
\end{lemma}

\begin{proof}
If $\lam\not= 0$ and $\lam e-ab$ is invertible then a direct
calculation shows that
\[
c(\lam e-ba)=(\lam e-ba)c=e
\]
where
\[
c=\lam^{-1}(e+b(\lam e-ab)^{-1}a).
\]
The first part of the lemma follows by interchanging the roles of
$a$ and $b$.

In order to prove the second part we have to show that $ab$ is
invertible if and only if $ba$ is invertible, whenever
$\dim(\cA)<\infty$. If $ab$ is invertible then there exists
$c\in\cA$ such that $(ab)c=e$. Therefore $L_aL_bL_c=I$ so
\[\det(L_a)\det(L_b)\det(L_c)=1.\] It follows that $L_a$, $L_b$ are
invertible. Lemma~\ref{lemmainv} now implies that $a,\, b$ are
invertible and hence that $ba$ is invertible. The converse is
similar.
\end{proof}

\begin{example}
The following arises in the theory of quantum groups, more
specifically the quantum torus. If $a,\, b\in\cA$ and $ab=zba$ for
some non-zero $z\in\F$ then Lemma~\ref{alginv} implies that
$\spec(ab)$ is invariant under multiplication by $z^{\pm 1}$. If
$\cA$ is a Banach algebra then the boundedness of the spectrum of
$ab$ further implies that $|z|=1$.
\end{example}

We say that $J\subseteq\cA$ is a right ideal if it is a linear
subspace of $\cA$ and $xy\in J$ whenever $x\in J$ and $y\in \cA$. If
$a\in\cA$ one defines the principal right ideal $J_a$ by $J_a=\{
ax:x\in\cA\}$ then $J_a=J_b$ if and only if there exist $x,y\in \cA$
such that $ax=b$ and $by=a$.

If $\cA$ is an algebra of linear maps on $V$ and $W$ is a linear
subspace of $V$, then one may define a right ideal of $\cA$ by
\[
J_W=\{ x: x(V)\subseteq W \} .
\]

If $\lam\in\F$ lies in the spectrum of $a\in\cA$ then one may define
a class of `Jordan' right ideals $J(a,\lam, n)$ by
\[
J(a,\lam, n)=\{ x\in \cA : (a-\lam e)^n x=0\}.
\]

\begin{lemma}
Under the above assumptions
\[
J(a,\lam, n)\subseteq J(a,\lam, n+1)
\]
for all natural numbers $n$. Either all $J(a,\lam, n)$ are distinct
or there exists $N$ such that $J(a,\lam, n)\not= J(a,\lam, n+1)$ if
$n<N$ but $J(a,\lam, n)= J(a,\lam, n+1)$ if $n\geq N$. If
$\dim(\cA)<\infty$ then the latter case occurs.
\end{lemma}

\begin{lemma}
If $a,b\in\cA$ and $\lam\not= 0$ then there is a natural isomorphism
of $J(ab,\lam, n)$ with $ J(ba,\lam, n)$ for all $n$.
\end{lemma}

\begin{proof}
We start by putting $J_1=J(ba,\lam,n)$ and $J_2=J(ab,\lam,n)$. If
$x\in J_1$ then
\[
(ab-\lam e)^n ax=a(ba-\lam e)^n x=0
\]
so $ax\in J_2$. Similarly $x\in J_2$ implies that $bx\in J_1$. Now
put
\[
c=\sum_{r=1}^n \frac{n!}{r!(n-r)!}a(ba)^{r-1}(-\lam)^{n-r}.
\]
By applying the two results just proved one sees that $x\in J_1$
implies $cx\in J_2$. Since $(ba-\lam e)^n=bc+(-\lam )^ne$, it
follows that $b(cx)=-(-\lam)^nx$ for all $x\in J_1$. Therefore
$L_b:J_2\to J_1$ is onto. A similar argument based on $(ab-\lam
e)^n=cb+(-\lam )^ne$ yields  $c(by)=-(-\lam)^n y$ for all $y\in
J_2$. Therefore $L_b:J_2\to J_1$ is one-one.

We have now proved that $L_b$ provides an isomorphism of $J_1$ with
$J_2$. In the particular case $n=1$ it proves that $\lam\not=0$ is
an eigenvalue of $ab$ if and only if it is an eigenvalue of $ba$.
\end{proof}

If $\dim(\cA)<\infty$ and we put $d(a,\lam,n)=\dim(J(a,\lam,n))$,
then the above lemma establishes that
\[
d(ab,\lam,n)=d(ba,\lam,n)\mbox{ for all }\lam\not= 0, \,\, n\geq 1.
\]
If $\lam=0$ these identities do not hold even if $\cA$ is the
algebra of $2\times 2$ matrices and
\[
a=\mtrx{0&1\\ 0&0} ,\hspace{2em} b= \mtrx{1&0\\ 0&0} .
\]

\section{Transformations of Polynomial Pencils}

Let $SL(2,\F)$ denote the multiplicative group of matrices
$\txtmtrx{a&b\\c&d}$ with entries in the field $\F$ and determinant
$1$. There is an extensive literature on such groups \cite{carter},
particularly when $\F$ is finite, which we do not assume. It is easy
to show that $SL(2,\F)$ is generated by $\txtmtrx{0&1\\
-1&0}$ and the set of all $\txtmtrx{1&x\\ 0&1}$ where $x\in\F$; see
\cite[Lemma~6.1.1]{carter}. From this it easily follows that the
centre of $SL(2,\F)$ consists of $\pm\txtmtrx{1&0\\ 0&1}$. Therefore
the centre contains two elements unless $\F$ has characteristic $2$,
in which case it contains only one.

We define $\cP_n$ to be the space of polynomials with degree at most
$n$ in the variable $x\in\F$, where the coefficients lie in the
algebra $\cA$. In the theory of non-linear eigenvalue problems the
spectrum of $p\in \cP_n$ is defined to be the set
\[
\spec(p)=\{ x\in \F:p(x)\mbox{ is not invertible in } \cA\}.
\]
We also include a symbolic element $\infty$ in $\spec(p)$ if
$p(s)=\sum_{r=0}^n a_rs^r$ and $a_n$ is not invertible in $\cA$.
Clearly $\spec(p)=\spec(bpc)$ whenever $b,\, c\in\cA$ are both
invertible.

\begin{lemma}
Given $g\in G=SL(2,\F)$ and $x\in \tilde{\F}:=\F\cup\{\infty\}$, the
formula
\[
g.x=\frac{ax+b}{cx+d}
\]
defines a map on $\tilde{\F}$ such that $g.(h.x)=(gh).x$ for all
$g,\, h\in G$ and all $x\in \tilde{\F}$. Moreover the formula
\[
(T_gp)(x)=(cx+d)^np(g^{-1}x)
\]
defines a linear map on $\cP_n$ such that $T_gT_hp=T_{gh}p$ for all
$g,h\in G$ and $p\in\cP_n$.
\end{lemma}

\begin{proof}
This is a direct computation.
\end{proof}

If $p(x)=x^na_n+\ldots +xa_1+a_0$ and $g=\txtmtrx{0&1\\-1&0}$ then
\[ (T_gp)(x)=x^na_0+\ldots +xa_{n-1}+a_n
\]
is often called the reversal of $p$.

\begin{lemma} If $g\in G=SL(2,\F)$, $p\in \cP_n$ and $q=T_g(p)$ then
\[
\spec(q)=g.\spec(p).
\]
\end{lemma}

\begin{lemma}
If $\F$ is infinite, $p\in\cP_n$ is regular in the sense that $p(x)$
is invertible for some $x\in\cA$ and $\dim(\cA)<\infty$ then there
exists $g\in G$ such that $T_g(p)=q$ where $q(s)=\sum_{r=0}^n a_r
s^r$ and $a_0,\, a_n$ are both invertible.
\end{lemma}

\begin{proof}
An argument involving determinants proves that the number of $x\in
\cA$ for which $p(x)$ is not invertible is no larger than $mn$ where
$m$ is the dimension of $\cA$ and $n$ is the degree of $p$. If
$x\not= y$ and $p(x),\, p(y)$ are both invertible, then there exists
$g\in G$ such that $g.x=0$ and $g.y=\infty$. The associated
polynomial $q=T_g(p)$ then has the required properties.
\end{proof}

In many respects it is more natural to consider homogeneous
polynomial pencils, i.e. functions
\[
p(x,y)=\sum_{r=0}^n x^ry^{n-r} a_r
\]
where $x,\, y \in\F$ and $a_r\in\cA$ for all $r$. The group
$SL(2,\F)$ acts on such polynomials in a similar fashion, but one no
longer needs to consider points at infinity separately. If $\F=\R$
then such a homogeneous polynomial is determined by its values on
the unit circle in $\R^2$. However, it is possible to analyze the
eigenvalue problem in the projective space $P(\F^2)$ directly,
without taking a particular parametrization of $P(\F^2)$ by points
in $\F\cap\{ \infty\}$; see \cite{DT}.

\section{Factorization of Polynomial Pencils}

Given a monic polynomial pencil with values in $\cA$, i.e. an
expression of the form
\[
p(x)=\sum_{r=0}^m a_r x^r
\]
where $a_r\in\cA$ for all $r$, $a_m=e$ and $x\in\F$, there may or
may not exist a factorization
\[
p(x)=\prod_{r=1}^m (xe-b_r)
\]

such that $b_r\in\cA$ for all $r$. Putting $m=2$ suppose that
$p(x)=x^2e-xa+b$ for some $a,\, b\in\cA$ and all $x\in \F$. If $p$
can be factorized as
\begin{equation}
p(x)=(xe-c)(xe-d)\label{quadfact}
\end{equation}
for some $c,\, d\in\cA$ and all $x\in\F$, a routine calculation
shows that
\begin{equation}
c^2-ca+b=0.\label{quad}
\end{equation}
Conversely if this equation is soluble then $p$ can be factorized.
See \cite[p. 75]{GLR}. Note also that (\ref{quad}) is equivalent to
\[
d^2-ad+b=0.
\]

The following example shows that (\ref{quad}) may not be soluble
even if $\F=\C$ and $\dim(\cA)<\infty$.

\begin{example}\label{nonfact}
Let $\cA$ be the algebra of $n\times n$ complex matrices and let
$a\in\cA$ be the elementary Jordan matrix such that $a_{r,s}=1$ if
$s=r+1$ and $a_{r,s}=0$ otherwise. Then $a^{n-1}\not= 0$ but
$a^{n}=0$. Suppose that there exist $c,\, d\in\cA$ such that
\[
x^2e-a=(x e-c)(x e-d)
\]
for all $x\in\C$. Then $d=-c$ and $c^2=a$. Therefore $c^{2n-2}\not=
0$ but $c^{2n}=0$. By writing down the Jordan form of $c$ one
readily sees that this is impossible.
\end{example}

If $\dim(\cA)<\infty$ then (\ref{quadfact}) and Lemma~\ref{alginv}
together imply that
\[
\spec(p)= \spec(c)\cup\spec(d).
\]
The following example shows that this need not be true if $\cA$ is
infinite-dimensional.

\begin{example}
This example requires some familiarity with the spectral theorem for
bounded self-adjoint operators. Let $c:\ell^2(\Z^+)\to \ell^2(\Z^+)$
be the bounded operator given by $(cf)(n)=f(n+1)$ for all $n\geq 0$
and let $d=c^\ast$. Then $cd=e$ but $dc\not=e$. One has
\[
p(x):=x^2e-ax+e=(xe-c)(xe-d)
\]
for all $x\in\C$, where $a=c+d$ is self-adjoint with
$\spec(a)=[-2,2]$. The spectrum of $p$ is $\{z:|z|= 1\}$ but $c,\,
d$ both have spectra equal to $\{z:|z|\leq 1\}$, so
\[
\spec(p)\not= \spec(c)\cup\spec(d).
\]
However, there exists a different factorization
\[
p(x)=(xe-u)(xe-u^\ast)
\]
where $u$ is unitary and $u,\, u^\ast, \, a$ all commute. For this
factorization
\[
\spec(p)= \spec(u)=\spec(u^\ast)=\{z:|z|=1\}.
\]
\end{example}

Given an algebra $\cA$, let $\cA_m$ denote the algebra of all
$m\times m$ matrices with entries in $\cA$ and let $e_m$ denote the
identity element in $\cA_m$.

\begin{lemma}
Let $p(x)=\sum_{r=0}^{m}x^ra_r$ where $x\in\F$, $a_r\in\cA$ for all
$r$ and $a_m=e$. Then $x$ lies in $\spec(p)$ if and only if it lies
in the spectrum of the matrix $X\in\cA_m$ defined by
\[
X_{r,s}=
\begin{choices}
1&\mbox{ if } s=r+1\\
-a_{s-1}&\mbox{ if } r=m\\
0&\mbox{ otherwise.}
\end{choices}
\]
\end{lemma}
\begin{proof}
This follows directly from the standard identity
\begin{equation}
G(x)(xe_m -X) =\mtrx{p(x)&0\\0&e_{m-1}}H(x) \label{4matrices_m}
\end{equation}
in which $G(x),\, H(x)\in \cA_m$ are invertible for all $x\in\F$. If
$m=4$ they are given by
\begin{eqnarray*}
G(x)&=&\left( \begin{array}{cccc}
x^3e+x^2a_3+xa_2+a_1&x^2e+xa_3+a_2&xe+a_3&e\\
-e&0&0&0\\
0&-e&0&0\\
0&0&-e&0
\end{array}\right)
\end{eqnarray*}
and
\begin{eqnarray*}
H(x)&=&\left( \begin{array}{cccc}
e&0&0&0\\
-x e&e&0&0\\
0&-x e&e&0\\
0&0&-x e&e
\end{array}\right),
\end{eqnarray*}
from which the general formula can be inferred.
\end{proof}

Our next two results will be used in the proof of
Theorem~\ref{polyfact_m}.

\begin{lemma}\label{Euclid}(Euclid)
Let
\[
p(x)=\sum_{r=0}^{m} a_rx^r
\]
for all $x\in \F$, where $a_r\in\cA$ for all $r$ and $a_m=e$. Given
$d\in\cA$ put
\[
q(x)=\sum_{r=0}^{m-1} b_rx^r
\]
where $b_r$ are defined inductively by $b_{m-1}=e$ and
\begin{eqnarray*}
a_{m-1}&=& b_{m-2}-b_{m-1}d\\
a_{m-2}&=& b_{m-3}-b_{m-2}d\\
\ldots&&\\
a_{1}&=& b_{0}-b_{1}d.
\end{eqnarray*}
Then
\[
p(x)=q(x)(xe-d) + {\rm rem}
\]
for all $x\in\F$ where
\[
{\rm rem}=a_0+b_0d=\sum_{r=0}^{m} a_rd^r.
\]
\end{lemma}

The proof involves simple substitutions, as does that of the
corollary below. We continue with the notation of
Lemma~\ref{Euclid}.

\begin{corollary}\label{pqdcor}
Let $\cA$ be an algebra of linear operators on some vector space
$V$. Let $\mu\in\F$ and $v\in V$. If $dv=\mu v$ and $p(\mu)v=0$ then
\[
p(x)v=q(x)(xe-d)v
\]
for all $x\in\F$.
\end{corollary}

From this point on we assume that $\cA$ is the algebra of all
$n\times n$ matrices over a field $\F$ and that $P(x)=
\sum_{r=0}^{m} x^rA_r$ where $A_m$ is the identity matrix $I$. The
following theorem may be found in \cite[Th.~3.21, Cor.~3.22]{GLR}.

\begin{theorem}\label{polyfact_m}
If the polynomial $\det(P(x))$ has $mn$ distinct roots in $\F$ then
there exist $C_1,\ldots\, C_m\in\cA$ such that
\[
P(x)=(xI-C_1)(xI-C_{2})\ldots (xI-C_m)
\]
for all $x\in \F$.
\end{theorem}

\begin{proof}
By calculating the determinants of every matrix in
(\ref{4matrices_m}), one sees that the hypothesis is equivalent to
the assumption that the $mn\times mn$ matrix $X$ has $mn$ distinct
eigenvalues $\lam_1,\ldots, \lam_{mn}$. The corresponding (right,
column) eigenvectors $w_r\in \F^{mn}$ of $X$ form a basis in
$\C^{mn}$.

If we write the eigenvector $w_r$ in the form $(u_{1,r},\ldots,
u_{m,r})'$, where each $u_r\in\F^n$, then (\ref{4matrices_m})
implies that
\[
\mtrx{P(\lam_r)&0\\0&e_{m-1}}H(\lam_r)w_r=0
\]
so
\[
P(\lam_r)u_{1,r}=0\mbox{ for all } r\in\{1,\ldots, mn\}.
\]
If $u\in\F^n$ then using the expansion

\[
\vctr{u\\ 0\\ \vdots\\ 0}=\sum_{r=1}^{mn} \alp_r\vctr{u_{1,r}\\
u_{2,r}\\ \vdots\\ u_{m,r}}
\]
we deduce that $u=\sum_{r=1}^{mn} \alp_r u_{1,r}$. Therefore $\{
u_{1,r}\}_{r=1}^{mn}$ spans $\F^n$. This sequence must contain a
basis of $\F^n$ and, after relabelling, we may donate it by $\{
u_{1,r}\}_{r=1}^{n}$.

Let $C_m$ denote the matrix such that $C_mu_{1,r}=\lam_r u_{1,r}$
for all $r$ such that $1\leq r \leq n$. Corollary~\ref{pqdcor}
implies that
\[
P(x)=Q(x)(xI-C_m)u_{1,r}
\]
for all $1\leq r\leq n$. Therefore
\[
P(x)=Q(x)(xI-C_m).
\]
The proof is completed by an induction, noting that $Q(x)$ satisfies
the same hypotheses as $P(x)$, but with $m$ replaced by $m-1$.
\end{proof}

\section{Quadratic equations in rings}

In this section we study quadratic polynomials in which the variable
$s$ or $x$ lies in a (generically non-commutative) ring $\cR$ rather
than in some base field. If $a,\, b,\, c,\, d\in \cR$ the algebraic
Riccati equation
\begin{equation}
sds+as+sb+c=0\label{ricc}
\end{equation}
is of great importance in a variety of contexts, but particularly in
optimal control theory; see \cite[Chap.~13]{ZDG}, where it is
assumed that $c,\, d$ are symmetric real matrices and that $a=b^{\rm
T}$. If $d$ is invertible (which is not generally assumed for the
Riccati equation) and we put $x=ds$, then (\ref{ricc}) may be
rewritten in the form $x^2+ux+xv+w=0$ where $u=dad^{-1}$, $v=b$ and
$w=dc$. We will study this equation within a general ring $\cR$.

\begin{theorem}\label{noncommpoly}
Given $u,\, v,\, w,\, u^\pr,\, v^\pr,\, w^\pr\in\cR$, the identity
\begin{equation}
x^2+ux+xv+w=x^2+u^\pr x+xv^\pr+w^\pr\label{noncommpoly1}
\end{equation}
holds for all $x\in\cR$ if and only if
\begin{equation}
w=w^\pr  \,\, \mbox{ \upshape{and} }\,\,
u-u^\pr=v^\pr-v\in\cC.\label{noncommpoly2}
\end{equation}
\end{theorem}

\begin{proof}
The proof of (\ref{noncommpoly1}) given (\ref{noncommpoly2}) is a
matter of direct substitution. Conversely, given
(\ref{noncommpoly1}), by putting $x=0$ and then $x=e$ one obtains
$w=w^\pr$ and $u+v=u^\pr +v^\pr$. It only remains to prove that
$c:=u-u^\pr\,\, (=v^\pr-v)$ lies in $\cC$. Given what we have just
proved (\ref{noncommpoly1}) can be reduced to $ux+xv=u^\pr x+xv^\pr$
and then rewritten in the form $(u-u^\pr)x=x(v^\pr-v)$. This implies
that $cx=xc$ for all $x\in\cA$, so $c\in\cC$.
\end{proof}

Let $p:\cR\to\cR$ be a monic quadratic polynomial of the form
\[
p(x)=x^2+ux+xv+w.\label{ringxpoly1}
\]
We define a factorization of $p$ to be an identity of the form
\begin{equation}
p(x)=(x-a)(x-b),\label{ringxpoly2}
\end{equation}
valid for all $x\in\cR$, where $a,\, b$ is an ordered pair in $\cR$.
In general a factorization need not exist, and if one does it need
not be non-unique.

\begin{example}
Let
\[
p(x)=x^2-\mtrx{1&0\\ 0&0}x-x\mtrx{0&0\\ 0&1} -\mtrx{2&0\\ 0&2}
\]
for all $x\in \cA=M(2,\C)$. A direct calculation shows that $p(a)=0$
for $a=\txtmtrx{ 2&0 \\ 0&2}$. If $p(x)=(x-a)(x-b)$ for some
$b\in\cA$ and all $x\in\cA$ then by putting $x=0$ one sees that
$b=-\txtmtrx{1&0\\ 0&1}$. By evaluating the other coefficients we
deduce that such a factorization does not exist.
\end{example}

\begin{theorem}
If the centre $\cC$ of $\cR$ is a field and $p$ has the
factorization (\ref{ringxpoly2}) then that factorization is unique
unless $c:=a-b\in\cC$, in which case one also has $p(x)=(x-b)(x-a)$
for all $x\in\cR$. There is no other factorization.
\end{theorem}

\begin{proof}
If $c\in\cC$ then it follows immediately that $a$ and $b$ commute
and that the second factorization in valid.

Conversely suppose that in addition to (\ref{ringxpoly2}) one has
$p(x)=(x-a^\pr)(x-b^\pr)$ for all $x\in \cR$. On expanding both
sides of the two factorizations a direct application of
Theorem~\ref{noncommpoly} yields
\[
a-a^\pr=b^\pr-b\in\cC \,\, \mbox{ \upshape{and} }\,\, ab=a^\pr
b^\pr.
\]
If one puts $c=a-a^\pr (\in \cC)$ then these identities imply that
$a^\pr=a-c$, $b^\pr=b+c$ and $c(a-b)-c^2=0$. Assuming that the two
factorizations are indeed distinct, $c\not= 0$ and multiplying by
$c^{-1}$ yields $a-b=c$. Hence $a^\pr=b$ and $b^\pr=a$.
\end{proof}

\section{The equation $ax-xb=c$}

This section treats a generalization of the Sylvester equation for
matrices, conventionally written in the form $AX+XB=C$, and the
continuous Lyapunov equation for matrices $AX+XA^{\rm T}+Q=0$, where
the superscript ${\rm T}$ denotes the matrix transpose. We consider
the problem at an algebraic level, but the main novelties are that,
motivated by the examples described in Section~\ref{analytic}, we do
not assume that the algebra is finite-dimensional or that the base
field is algebraically closed.

If $a,\, b,\, c$ lie in the algebra $M(n,\F)$ of $n\times n$
matrices with entries in a field $\F$, the equation $ax-xb=c$ is not
always soluble for $x$ within $M(n,\F)$. Lemma~\ref{constraints}
describes some constraints on the $c$ for which the equation is
$ax-xa=c$ is soluble. A detailed analysis of the conditions under
which a matrix equation in $M(n,\C)$ of the form $ax-xb=c$ is
soluble is given in \cite[Chap.~S2]{GLR} and \cite{FW, Roth, ZDG}.
Our own discussion considers the case in which $a,b,c,x$ all lie in
an algebra $\cA$ over a general field $\F$.

\begin{lemma}\label{pqf}
Given $a,\, b\in\cA$ the following are equivalent.
\begin{enumerate}
\item
$a,\, b$ are algebraic and have relatively prime minimum polynomials $p,\, q$.
\item
There exists $f\in\cP$ such that $f(a)=e$ and $f(b)=0$.
\end{enumerate}
\end{lemma}

\begin{proof}
1$\implies$2. There exist $h,\,k\in\cP$ such that $ph+qk=1$. If we
put $f=qk$ then $f(b)=q(b)k(b)=0$ and $f(a)=e-p(a)h(a)=e$.

2$\implies$1. The identity $f(a)=e$ implies that $f\not=0$ as an
element of $\cP$, and $f(b)=0$ then implies that $b$ is algebraic
with a minimum polynomial $q$; moreover there exists $k\in\cP$ such
that $f=qk$. Putting $g=1-f$ a similar argument implies that $a$ is
algebraic with a minimum polynomial $p$ and that $g=ph$ for some
$h\in\cP$. Since $ph+qk=1$, $p,\, q$ are relatively prime.
\end{proof}

We say that $a,\, b\in\cA$ are \emph{spectrally disjoint} if they
are both algebraic and the conditions of Lemma~\ref{pqf} hold.

\begin{theorem}\label{commutator}
If $a,\, b\in\cA$ are spectrally disjoint then $ax-xb=c$ has a
unique solution $x\in\cA$ for all $c\in\cA$.
\end{theorem}

\begin{proof}
Let $\cB$ be the algebra of all linear maps from $\cA$ to $\cA$.
Define $A,\, B\in\cB$ by $A(x)=ax$ and $B(x)=xb$. If $f\in\cP$ is
defined as in Lemma~\ref{pqf} then $f(A)(x)=f(a)x=x$ and
$f(B)(x)=xf(b)=0$ for all $x\in \cA$. Therefore $f(A)=I$ and
$f(B)=0$ as elements of $\cB$. Since $A,\, B$ commute we have
\begin{equation}
I=f(A)-f(B)=(A-B)g(A,B)=g(A,B)(A-B)\label{A-Binverse}
\end{equation}
where $g$ is a polynomial in two commuting variables; this may be
proved by using the particular case
\[
A^n-B^n=(A-B)(A^{n-1}+A^{n-2}B\ldots + AB^{n-2}+B^{n-1}).
\]
Equation~(\ref{A-Binverse}) implies that $A-B$ is invertible as an
element of $\cB$, or equivalently one-one and onto as a linear
transformation on $\cA$.
\end{proof}

\begin{corollary}
Let $\cJ$ be an ideal in $\cA$. Under the assumptions of
Theorem~\ref{commutator} one has $x\in\cJ$ if and only if $c\in\cJ$.
\end{corollary}

\begin{proof}
The simplest proof involves applying Theorem~\ref{commutator} to the
relevant elements of $\cA/\cJ$. The following is a more constructive
approach. The statement $(x\in\cJ)\implies(c\in\cJ)$ follows
directly from the defining properties of an ideal. Conversely
(\ref{A-Binverse}) implies
\[
x=(A-B)^{-1}c=g(A,B)c.
\]
Since $A$ and $B$ map $\cJ$ into $\cJ$, one immediately concludes
that $c\in\cJ$ implies $x\in\cJ$.
\end{proof}

Theorem~\ref{commutator} can be extended to bounded operators on a
complex Banach space or even to elements of a Banach algebra; see
\cite{rosenblum}. The following is a more general version of
Theorem~\ref{commutator}.

\begin{theorem}\label{commuter}
If $a,\, b\in\cA$ commute and are spectrally disjoint then $a-b$ is
invertible.
\end{theorem}

\begin{proof}
Let $f\in\cP$ satisfy $f(a)=e$ and $f(b)=0$. There exists a
polynomial $g$ in two commuting variables $x,\, y$ such that
\[
f(x)-f(y)=(x-y)g(x,y) \mbox{ for all } x,\, y\in \F.
\]
This implies that
\[
e=f(a)-f(b)=(b-a)g(a,b)
\]
so $a-b$ is invertible with inverse $g(a,b)$.
\end{proof}

\begin{example}
Let $V$ denote the vector space of all $m\times n$ matrices with
entries in the field $\F$. Given $c\in V$, $a\in M(m,\F)$ and $b\in
M(n,\F)$, one might wish to solve the equation $ax-xb=c$ for $x\in
V$. This problem may be analyzed by using Theorem~\ref{commuter},
where $\cA$ is taken to be the algebra of all linear maps from $V$
to $V$. If $L,\, R \in \cA$ are defined by $L(v)=av$ and $R(v)=vb$
then one sees immediately that $L$ and $R$ commute. Moreover $L$ has
the same spectrum and minimum polynomial as $a$, while $R$ has the
same spectrum and minimum polynomial as $b$. Therefore the equation
$ax-xb=c$ is uniquely soluble if $a$ and $b$ are spectrally
disjoint.
\end{example}

In order to apply Theorem~\ref{commutator} or \ref{commuter} in a
purely algebraic context we need two lemmas. We define $\cP^\ast$ to
be the ring of polynomials with entries in the algebraic closure
$\F^\ast$ of $\F$. We put $\cA^\ast=\cA\otimes\F^\ast$ and identify
$a\in\cA$ with $a\otimes 1 \in\cA^\ast$ where $1$ is the
multiplicative identity in $\F^\ast$. We define $\cB$ to be the
algebra of all $\F$-linear operators from $\cA$ to $\cA$ and
$\cB^\ast$ to be the algebra of all $\F^\ast$-linear operators from
$\cA^\ast$ to $\cA^\ast$.

\begin{lemma}\label{oneminpoly}
Let $\cA$ be an algebra over $\F$ and let $a\not= 0$ be an algebraic
element of $\cA$. The following four polynomials coincide:
\begin{enumerate}
\item the minimum polynomial $p_1$ of $a\in\cA$;
\item the minimum polynomial $p_2$ of $L_a\in\cB$;
\item the minimum polynomial $p_3$ of $a\in\cA^\ast$;
\item the minimum polynomial $p_4$ of $L_a\in\cB^\ast$.
\end{enumerate}
\end{lemma}

\begin{proof}
The equivalence of items 1 and 2  uses the fact that the
homomorphism $a\to L_a$ is one-one. Similarly for items 3 and 4. We
prove the equivalence of items 1 and 3 by using the fact that $\cP$
and $\cP^\ast$ are principal ideal domains.

Since $p_1(a)=0$ in $\cA$, this also holds in $\cA^\ast$, so $p_3$
is a factor of $p_1$. Since both are monic, in order to establish
that they are equal we need only prove that $\deg(p_3)\geq
\deg(p_1)$. Let $P$ be a projection on $\F^\ast$ with range $\F$,
where $\F^\ast$ is regarded as a vector space over $\F$. The
polynomial $p_5=P(p_3)\in\cP$ is monic with $p_5(a)=0$. Therefore
$p_1$ is a factor of $p_5$ and $\deg(p_3)=\deg(p_5)\geq \deg(p_1)$.
\end{proof}

\begin{lemma}
Let $\cA$ be an algebra over a field $\F$ and let $a,\, b$ be two
algebraic elements of $\cA$. If $p,\, q$ are their minimum
polynomials then the following are equivalent.
\begin{enumerate}
\item
$p,\, q$ are relatively prime in $\cP$.
\item
$p$ and $q$ are relatively prime in $\cP^\ast$.
\item
$L_a$ and $L_b$ have disjoint spectra regarded as operators on
$\cA^\ast$.
\end{enumerate}
\end{lemma}

\begin{proof}
$1\implies 2.$ The existence of an identity $h p+k q=1$ in $\cP$
implies that the same identity holds in $\cP^\ast$.

$2\implies 1.$ Since $p,\, q$ are relatively prime in $\cP^\ast$,
there exist $h^\ast,\, k^\ast\in\cP^\ast$ such that $h^\ast p+k^\ast
q=1$. Defining the projection $P$ as in the proof of
Lemma~\ref{oneminpoly} and putting $h=P(h^\ast),\, k=P(k^\ast)$ we
obtain $hp+kq=1$; hence $p,\, q$ are relatively prime in $\cP$.

$2\Leftrightarrow 3.$ It follows from Theorem~\ref{alginvert} that
both statements are equivalent to $p$ and $q$ having no common
zeros.
\end{proof}

\begin{example}
Let $\cQ$ denote the algebra of quaternions over a base field $\F$.
If
\[
a=a_0+a_1i+a_2j+a_3k
\]
then the minimum polynomial of $a$ is
\[
p(z)=z^2-2a_0z+(a_0^2+a_1^2+a_2^2+a_3^2).
\]
The field $\F$ is said to be formally real if whenever $a_r\in\F$
for $1\leq r\leq n$ and $\sum_{r=1}^n a_r^2=0$ one has $a_r=0$ for
all $\leq r \leq n$. Examples are $\R$ and the field of all rational
functions in one variable whose coefficients are all real.

If $\F$ is formally real then $\cQ$ is a division algebra (skew
field). In this case $\spec(a)=\emptyset$ unless $a_1=a_2=a_3=0$, in
which case $\spec(a)=a_0$. Given $a,\, b\in \cQ$ the equation
$ax-xb=c$  has a unique solution $x\in\cQ$ for all $c\in\cA$ if and
only if either $a_0\not= b_0$ or $a_1^2+a_2^2+a_3^2\not =
b_1^2+b_2^2+b_3^2$.
\end{example}

If $a=b$ then Theorem~\ref{commutator} cannot be applied. The next
two results provide some insight into this case. A trace on an
algebra $\cA$ is by definition a non-zero linear functional
$\tr:\cA\to\F$ such that $\tr(ab)=\tr(ba)$ for all $a,\, b\in\cA$.
Important infinite-dimensional examples arise in the theory of
finite von Neumann algebras and free probability theory
\cite{tao,voic}, as well as Example~\ref{tensor}.

\begin{lemma}\label{constraints}
Let $\cA$ be an algebra with a trace $\tr$. If $a,\, x,\, c\in\cA$
and $ax-xa=c$ then $c$ satisfies the constraints $\tr(a^mc)=0$ for
all $m\geq 0$.
\end{lemma}

\begin{proof}
One takes the trace of both sides of the identity
\[
[a^{m+1},x]=a^m[a,x]+a^{m-1}[a,x]a+\ldots +a[a,x]a^{m-1}+[a,x]a^m.
\]
\end{proof}

The constraints on $c$ in Lemma~\ref{constraints} are not
necessarily independent for different $m$.

\begin{theorem}
Let $\cA=M(n,\F)$ and let $d$ be the dimension of the set of $C$ for
which $AX-XA=C$ is soluble. Then $0\leq d\leq n^2-n$. The case $d=0$
occurs if $A=I$ while the case $d=n^2-n$ occurs if $A$ is the
$n\times n$ elementary Jordan matrix $J_{n,\lam}$ defined by
\[
J_{n,\lam,r,s}=\begin{choices}
\lam&\txtif s=r,\\
1&\txtif  s=r+1,\\
0&\txtotherwise.
\end{choices}
\]
\end{theorem}

\begin{proof}
We define $L:\cA\to\cA$ by $L(X)=AX-XA$. Since the rank of a linear
transformation does not alter if one increases the base field, we
may assume that $\F$ is algebraically closed. We have to prove that
\[
\dim(\ker(L))\geq n,
\]
with equality in the stated case. We choose a basis of $\F^n$ for
which $A$ is a block diagonal matrix whose blocks are elementary
Jordan matrices. More precisely we assume that the diagonal blocks
are $J_{n_s,\lam_s}$ where $1\leq s\leq S$, so that $\sum_{s=1}^S
n_s=n$. Since $\ker(L)$ contains the space of all diagonal block
matrices $X$ such that $XA=AX$, we need only prove that the space
$V_s$ of $n_s\times n_s$ matrices $X_s$ satisfying
\[
X_sJ_{n_s,\lam_s}=J_{n_s,\lam_s}X_s
\]
has $\dim(V_s)=n_s$. A direct calculation shows that $V_s$ is the
space of $n_s\times n_s$ upper triangular Toeplitz matrices, which
has the required dimension.
\end{proof}

\section{Fields of analytic functions}\label{analytic}
The ring $\cP$ of all polynomials with complex coefficients is an
integral domain and its quotient field is the space $\cQ$ of all
rational functions on $\C$. $\cQ$ is not algebraically closed. If
$a,\, b,\, c\in\cQ$ and $a\not=0$ then the quadratic equation
$ax^2+bx+c=0$ has a solution in $\cQ$ if and only if the
discriminant $\Delta=b^2-4ac$ has a square root in $\cQ$. This holds
if and only if every zero and pole of $f$ has even order.

The elements of $\cQ$ have a finite number of poles and so, strictly
speaking, are not functions. The proper description involves the
theory of germs of analytic functions. To define these we start with
expressions of the form
\begin{equation}
f(z)=z^m\sum_{r=0}^\infty a_r z^r\label{fgerm}
\end{equation}
where $m\in\Z$ and $a_0\not= 0$ (unless $f$ is identically zero); we
assume that the series converges in some neighbourhood $U$ of $0$.
We identify two `germs' $(f,U)$ and $(g,V)$ if $g=f$ on $U\cap V$,
and define the sum and product of two germs by pointwise operations,
excluding the origin. The set of such germs is a field $\F_1$ that
contains $\cQ$. The series in (\ref{fgerm}) has a radius of
convergence $R\in (0,\infty]$, so the corresponding germ has a
canonical representative in which one puts $U=\{ z:|z|<R\}$. The
ring $\E$ of entire functions on $\C$ is contained in $\F_1$ and the
factorization theorem of Weierstrass, \cite[Section~13.3]{NP},
implies that its quotient field, also contained in $\F_1$, is the
field of (germs of) meromorphic functions.

The roots of a polynomial with coefficients in $\cQ$ may be described algebraically or geometrically, and the two perspectives are significantly different. Elements of the algebraic closure of F may be associated with a single-valued branch of a multi-valued analytic function, but geometrically speaking the natural entity is the Riemann surface associated with the multi-valued function. The following description presents the elements of the algebraic closure $\cF_2$ of $\cQ$ as analytic functions on a particular domain in $\C$; this domain is far from unique and assigns the origin in $\cC$ a status that it does not deserve, but it has the merit of being simple and explicit.

We define $\G$ to consist of all pairs $(f,U)$ where $f$ is an
analytic function on a  region $U$ of the form
\begin{equation}
U=\C\backslash \bigcup_r S(w_r)\label{rays}
\end{equation}
where each $S(w_r)$ is a semi-infinite ray starting at $w_r$ and the number of rays is finite or countable; in the second case it is assumed that $|w_r|\to\infty$ as $r\to \infty$. If $w=0$ we define $S(w)= [0,\infty)$ and for all other $w$ we define $S(w)=w[1,\infty)$. Each such `star-shaped' region $U$ is dense in $\C$ and simply connected, and the intersection of any finite number of such $U$ is of the same form.

The sum of $(f,U)$ and $(g,V)$ is defined to be the pointwise sum of
$f$ and $g$ with domain $U\cap V$, with a similar definition for the
product. We again identify $(f,U)$ and $(g,V)$ as elements of $\cG$
if $f=g$ on $U\cap V$. This turns $\G$ into an integral domain,
which has the following important property.

\begin{theorem}
Let
\[
p(z,w)=\sum_{r=0}^n a_r(w)z^r
\]
where $a_r$ are entire or meromorphic functions of $w\in\C$ and
$a_n=1$. Then
\[
p(z,w)=\prod_{r=1}^n(z-b_r(w))
\]
where each $b_r\in\G$.
\end{theorem}

\begin{proof}
According to \cite[Theorem~IV.14.2]{SZ} the roots of the polynomial
$z\to p(z,w)$ are branches of analytic functions of $w$ whose
singularities are all algebraic, and the singularities form a
discrete set in $\C$. If we denote these by $w_r$ then each root has
one or more single-valued branches on the set obtained from $\C$ by
removing the rays through all $w_r$.
\end{proof}

\begin{example}
The solutions of the equation
\[
z^3-(1-w^2)=0
\]
are the three functions
\[
f_r(w)=\alp^r (1-w^2)^{1/3}
\]
where $r=0,1,2$ and $\alp=\rme^{2\pi i/3}$. These are different
branches of the multi-valued cube root. One may define a
single-valued branch of $(1-w^2)^{1/3}$ on
\[
U=\C\backslash \left( [1,+\infty)\cup(-\infty,-1]\right).
\]
by analytically continuing the obvious power series from the unit
ball.
\end{example}

{\bf Acknowledgements} I would like to thank M. Breuning, W. J.
Harvey, J. R. Partington, A. Pushnitski and S. Richard for useful
comments.

Department of Mathematics\\
King's College London\\
Strand\\
London, WC2R 2LS\\
UK

E.Brian.Davies@kcl.ac.uk

\end{document}